\documentclass[12pt]{amsart}
 \usepackage{amscd,amsmath,amsthm,amssymb}
 \usepackage{pstcol,pst-plot,pst-3d}
 \psset{unit=0.7cm,linewidth=0.8pt,arrowsize=2.5pt 4}

\newpsstyle{fatline}{linewidth=1.5pt}
\newpsstyle{fyp}{fillstyle=solid,fillcolor=verylight}
\definecolor{verylight}{gray}{0.97}
\definecolor{light}{gray}{0.93}
\definecolor{medium}{gray}{0.82}
 \def\frk{\frak}               

 \def\mm{{\frk m}}
 
 \def\nn{{\frk n}}
 \def\Sc{{\mathcal S}}

 \def\xb{{\bold x}}

 \def\opn#1#2{\def#1{\operatorname{#2}}} 
 \opn\chara{char} \opn\length{\ell} \opn\pd{pd} \opn\rk{rk}
 \opn\projdim{proj\,dim} \opn\injdim{inj\,dim} \opn\rank{rank}
 \opn\depth{depth} \opn\grade{grade} \opn\height{height}
 \opn\embdim{emb\,dim} \opn\codim{codim}
 
 \opn\Tr{Tr} \opn\bigrank{big\,rank}
 \opn\superheight{superheight}\opn\lcm{lcm}
 \opn\trdeg{tr\,deg}
 \opn\reg{reg} \opn\lreg{lreg} \opn\ini{in} \opn\lpd{lpd}
 \opn\size{size} \opn\sdepth{sdepth}
 \opn\link{link}\opn\fdepth{fdepth}\opn\lex{lex}
 \opn\div{div} \opn\Div{Div} \opn\cl{cl} \opn\Cl{Cl}
 \opn\Spec{Spec} \opn\Supp{Supp} \opn\supp{supp} \opn\Sing{Sing}
 \opn\Ass{Ass} \opn\Min{Min}\opn\Mon{Mon}
 \opn\Ann{Ann} \opn\Rad{Rad} \opn\Soc{Soc}
 \opn\Im{Im} \opn\Ker{Ker} \opn\Coker{Coker} \opn\Am{Am}
 \opn\Hom{Hom} \opn\Tor{Tor} \opn\Ext{Ext} \opn\End{End}
 \opn\Aut{Aut} \opn\id{id}
 
 \opn\nat{nat}
 \opn\pff{pf}
 \opn\Pf{Pf} \opn\GL{GL} \opn\SL{SL} \opn\mod{mod} \opn\ord{ord}
 \opn\Gin{Gin} \opn\Hilb{Hilb}\opn\sort{sort}
 \opn\aff{aff} \opn
 \con{conv} \opn\relint{relint} \opn\st{st}
 \opn\lk{lk} \opn\cn{cn} \opn\core{core} \opn\vol{vol}
 \opn\link{link} \opn\star{star}\opn\lex{lex}\opn\set{set}
 \opn\gr{gr}
 
 \def\pot#1#2{#1[\kern-0.28ex[#2]\kern-0.28ex]}

 \opn\dirlim{\underrightarrow{\lim}}
 \opn\inivlim{\underleftarrow{\lim}}
 \let\sect=\cap
 \let\dirsum=\oplus

 \let\Union=\bigcup
 \let\Sect=\bigcap
 \let\Dirsum=\bigoplus
 
 \let\to=\rightarrow
 
 \def\Implies{\ifmmode\Longrightarrow \else
         \unskip${}\Longrightarrow{}$\ignorespaces\fi}
 \def\implies{\ifmmode\Rightarrow \else
         \unskip${}\Rightarrow{}$\ignorespaces\fi}
 \def\iff{\ifmmode\Longleftrightarrow \else
         \unskip${}\Longleftrightarrow{}$\ignorespaces\fi}

 \let\:=\colon
 \newtheorem{Theorem}{Theorem}[section]
 
 \newtheorem{Corollary}[Theorem]{Corollary}
 \newtheorem{Proposition}[Theorem]{Proposition}

 \newtheorem{Example}[Theorem]{Example}

 \let\epsilon\varepsilon
 \let\kappa=\varkappa
 \def\qed{\ifhmode\textqed\fi
       \ifmmode\ifinner\quad\qedsymbol\else\dispqed\fi\fi}
 \def\textqed{\unskip\nobreak\penalty50
        \hskip2em\hbox{}\nobreak\hfil\qedsymbol
        \parfillskip=0pt \finalhyphendemerits=0}
 \def\dispqed{\rlap{\qquad\qedsymbol}}
 
 \opn\dis{dis}
 \def\pnt{{\raise0.5mm\hbox{\large\bf.}}}
 
 \opn\Lex{Lex}

\begin{document}

 \title {Ordinary and symbolic powers are Golod}

 \author {J\"urgen Herzog and  Craig Huneke}

\address{J\"urgen Herzog, Fachbereich Mathematik, Universit\"at Duisburg-Essen, Campus Essen, 45117
Essen, Germany}
\email{juergen.herzog@uni-essen.de}

\address{Craig Huneke, Department of Mathematics, University of Virginia,
1404 University Ave,
Charlottesville,  VA 22903-2600,
United States}
\email{huneke@virginia.com}


 \begin{abstract}
 Let $S$ be a positively  graded polynomial ring  over a field of characteristic $0$, and $I\subset S$  a proper graded ideal. In this note it is shown that $S/I$ is Golod if $\partial(I)^2\subset I$. Here $\partial(I)$ denotes the ideal generated by all the partial derivatives of elements of $I$. We apply this result to find large classes of Golod ideals, including powers, symbolic powers, and saturations of ideals.
 \end{abstract}

\thanks{Part of the paper was written while the authors were visiting MSRI at Berkeley. They wish to acknowledge the support, the hospitality and the inspiring atmosphere of this institution. The second author was partially suppported by NSF grant 1259142.}

 \maketitle

 \section*{Introduction}
 Let $(R,\mm)$ be a Noetherian local  ring with residue class field $K$, or a standard graded $K$-algebra with graded maximal ideal $\mm$. The formal power series $P_R(t)= \sum_{i \geq 0} \dim_K \Tor_i^{R} (R/\mm,R/\mm) t^i$ is called the {\em Poincar\'{e}  series} of $R$. Though the ring is Noetherian, in general the Poincar\'{e}  series of $R$  is  not a rational function. The first example that showed that $P_R(t)$ is not
necessarily rational was given by Anik \cite{An}.  In the meantime more such examples are known, see \cite{Ro} and its references. On the other hand, Serre showed that $P_R(t)$ is coefficientwise bounded above by the rational series
 \[
 \frac{(1+t)^n}{1-t\sum_{i\geq 1}\dim_K H_i(\xb;R)t^i},
 \]
where $\xb=x_1,\ldots,x_n$ is a minimal system of generators of $\mm$ and where $H_i(\xb;R)$ denotes the $i$th Koszul homology of the sequence $\xb$.

The ring $R$ is called {\em Golod}, if $P_R(t)$ coincides with this upper bound given by Serre. There is also a relative version of Golodness which is defined for local homomorphisms as an obvious extension of the above concept of Golod rings. We refer the reader for details regarding Golod rings and Golod homomorphism to the survey article \cite{Av} by Avramov. Here we just want to quote the following result of Levin \cite{Le} which says that for any Noetherian local ring $(R,\mm)$,  the canonical map $R\to R/\mm^k$ is a Golod homomorphism for all $k\gg 0$. It is natural to ask whether in this statement $\mm$ could be replaced by any other proper ideal of $R$. Some very recent results indicate that this question may have a positive answer. In fact, in \cite{HWY} it is shown that if $R$ is regular, then for any proper ideal $I\subset R$ the residue class ring  $R/I^k$ is Golod for $k\gg 0$, which, since $R$ is regular, is equivalent to  saying that the residue class map $R\to R/I^k$ is a Golod homomorphism  for $k\gg 0$. But how big $k$ has to be chosen to make sure that $R/I^k$ is Golod? In the case that $R$ is the polynomial ring and $I$ is a proper monomial ideal, the surprising answer is that $R/I^k$ is Golod for all $k\geq 2$, as has been shown by Fakhari and Welker  in \cite{FW}. The authors show even more: if $I$ and $J$ a proper monomial ideals, then $R/IJ$ is Golod. Computational evidence suggests that $R/IJ$ is Golod  for any two proper ideals $I,J$ in a local ring (or graded ideals in a graded ring). This is supported by a result of Huneke \cite{Hu} which says that for an unramified regular local ring $R$, the residue class ring $R/IJ$ is never Gorenstein, unless $I$ and $J$ are principal ideals:  Indeed, being Golod implies in particular that the Koszul homology $H(\xb;R)$ admits trivial multiplication,  while for a Gorenstein ring, by a result of Avramov and Golod \cite{LAv},  the multiplication map induces for all $i$ a non-degenerate pairing $H_i(\xb;R)\times H_{p-i}(\xb;R) \to H_p(\xb:R)$ where $p$ is place of the top non-vanishing homology of the Koszul homology. In the case that $I$ and $J$ are not necessarily monomial ideals, it is only known that $R/IJ$ is Golod if $IJ=I\sect J$, see \cite{HSt}.

\medskip
In the present note  we consider graded ideals in the  graded polynomial ring  $S=K[x_1,\ldots,x_n]$ over a field $K$ of characteristic $0$ with $\deg x_i=a_i>0$ for $i=1,\ldots,n$. The main result of Section~\ref{diff} is given in  Theorem~\ref{muchbetter} which says  that $S/I$ is Golod if $\partial(I)^2\subset I$. Here $\partial(I)$ denotes the ideal which is generated by all partial derivatives of the generators of $I$. This is easily seen to be independent of the generators of $I$ chosen. We call an ideal {\em  strongly Golod} if  $\partial(I)^2\subset I$. In Section~\ref{classes} it is shown that the class of strongly Golod ideals is closed under several important ideal operations like products, intersections and  certain colon ideals. In particular it is shown that for any $k\geq 2$,  the $k$th power of a graded ideal, as well as its $k$th symbolic power and its  $k$th saturated power, is strongly Golod. We also prove the surprising fact that all  the primary components of a graded ideal $I$ which belong  to the minimal prime ideals of $I$ are  strongly Golod, if $I$ is so. Even more, we prove that every strongly Golod ideal has a primary decomposition in which each
primary component is strongly Golod. We are also able to prove that the integral closure of a strongly Golod monomial ideal
is again strongly Golod; in particular, the integral closure of $I^k$ is always strongly Golod if $I$ is monomial and $k\geq 2$. We
do not know if this last assertion is true for general ideals.

It should be noted that our results, though quite general, do not imply the result of Fakhari and Welker concerning products of monomial ideals. One can easily find products of monomial ideals which are not strongly Golod.  Moreover we have to require that the base field is of characteristic $0$. Actually as the proof will show, it is enough to require in Theorem~\ref{muchbetter} that the characteristic of $K$ is big enough compared with the shifts in the graded free resolution of the ideal.

A preliminary version of this paper by the first author was posted on ArXiv, proving that powers are Golod. After some discussions with the second author, this final
version emerged.

\section{A differential condition for Golodness}
\label{diff}
Let $K$ be a field of characteristic $0$, and $S=K[x_1,\ldots,x_n]$ the graded polynomial ring over $K$ with $\deg x_i=a_i>0$ for $i=1,\ldots,n$, and  let $I\subset S$ be a graded ideal different from $S$. We denote by $\partial(I)$ the ideal which is generated by the partial derivatives $\partial f/\partial x_i$ with  $f\in I$ and $i=1,\ldots,n$.

\begin{Theorem}
\label{muchbetter}
Suppose that $\partial(I)^2\subset I$. Then $S/I$ is Golod.
\end{Theorem}

\begin{proof}
We set $R=S/I$, and  denote by $K(R)$ the Koszul complex of $R$ with respect to the sequence  $\xb=x_1,\ldots,x_n$. Furthermore we denote by  $Z(R)$, $B(R)$ and $H(R)$  the module of cycles, boundaries and the homology of $K(R)$.

Golod \cite{Go} showed that Serre's upper bound for the Poincar\'{e} series
 is reached if and only if all Massey operations of $R$ vanish. By definition, this is the case (see \cite[Def. 5.5 and 5.6]{AKM}), if for each subset $\mathcal{S}$ of
  homogeneous elements of $\Dirsum_{i=1}^nH_i(R)$  there exists a function $\gamma$, which is defined on the set of finite
  sequences of elements from $\Sc$  with values in $\mm\dirsum\Dirsum_{i=1}^nK_i(R)$,  subject to the following conditions:
  \begin{enumerate}
    \item[(G1)] if $h\in \Sc$, then $\gamma(h)\in Z(R)$ and $h=[\gamma(h)]$;
    \item[(G2)] if  $h_1,\ldots,h_m$ is a sequence in $\mathcal{S}$ with $m>1$,  then
    \[
      \partial\gamma(h_1,\ldots,h_m)=\sum_{\ell=1}^{m-1}\overline{\gamma(h_1,\ldots,h_\ell)}\gamma(h_{\ell+1},\ldots,h_m),
    \]
  where $\bar{a} = (-1)^{i+1}a$ for $a\in K_i(R)$.
  \end{enumerate}

Note that (G2) implies,  that $\gamma(h_1)\gamma(h_2)$ is a boundary for all $h_1,h_2\in \Sc$ (which in particular implies that the Koszul homology of a Golod ring has trivial multiplication).  Suppose now that for each $\Sc$ we can choose a functions $\gamma$ such that $\gamma(h_1)\gamma(h_2)$ is not only  a boundary  but that $\gamma(h_1)\gamma(h_2)=0$ for all $h_1,h_2\in \Sc$.  Then obviously we may set $\gamma(h_1,\ldots h_r)=0$  for all $r\geq 2$, so that in this case (G2) is satisfied and $R$ is Golod.

The proof of Theorem~\ref{muchbetter} follows, once we have shown that $\gamma$ can be chosen that $\gamma(h_1)\gamma(h_2)=0$ for all $h_1,h_2\in \Sc$. For the proof of this fact we use the following result from \cite{H}: Let
\[
0\to F_p\to F_{n-1}\to \cdots \to F_1\to F_0\to S/J\to 0
\]
be the graded minimal free $S$-resolution of $S/J$, and for each $i$ let $f_{11},\ldots, f_{ib_i}$ a homogeneous basis of $F_i$. Let $\varphi_i\: F_i\to F_{i-1}$ denote the chain maps in the resolution, and let
\[
\varphi_i(f_{ij})=\sum_{k=1}^{b_{i-1}}\alpha_{jk}^{(i)}f_{i-1,k},
\]
where the $\alpha_{jk}^{(i)}$ are  homogeneous polynomials.

In \cite[Corollary 2]{H} it is shown that for all $l=1,\ldots,p$  the elements
\[
\sum_{1\leq i_1<i_2<\cdots <i_l\leq n}a_{i_1}a_{i_2}\cdots a_{i_l}\sum_{j_2=1}^{b_{l-1}}\cdots \sum_{j_l=1}^{b_1}c_{j_1,\ldots,j_l}\frac{\partial(\alpha_{j_1,j_2}^{(l)},\alpha_{j_2,j_3}^{(l-1)},\ldots, \alpha_{j_l,1}^{(1)})}{\partial(x_{i_1},\ldots,x_{i_l})}e_{i_1}\wedge \cdots \wedge e_{i_l},
\]
$j_1=1,\ldots,b_l$ are cycles of $K(R)$ whose homology classes form a $K$-basis of $H_l(R)$.

Thus we see that a $K$-basis of $H_l(R)$ is given by cycles which are linear combinations of Jacobians determined by the entries $\alpha_{jk}^{(i)}$  of the matrices describing the resolution of $S/J$. The coefficients $c_{j_1,\ldots,j_l}$ which appear in these formulas are rational numbers determined by the degrees of the $\alpha_{jk}^{(i)}$, and the elements $e_{i_1}\wedge \cdots \wedge e_{i_l}$ form the natural $R$-basis of the free module  $K_l(R)=\bigwedge^l(\Dirsum_{i=1}^nRe_i)$.

From this result it follows that any homology class of $H_l(R)$ can be represented by a cycle which is a linear combination of Jacobians of the form

\begin{eqnarray}
\label{jacobian}
\frac{\partial(\alpha_{j_1,j_2}^{(l)}\alpha_{j_2,j_3}^{(l-1)},\ldots, \alpha_{j_l,1}^{(1)})}{\partial(x_{i_1},\ldots,x_{i_l})}.
\end{eqnarray}
We choose such representatives  for the elements of the set $\Sc$. Thus we may choose the map $\gamma$ in such a way that it assigns to each element of $\Sc$ a cycle which is a linear combination of Jacobians as in  (\ref{jacobian}).

The elements $\alpha_{j_l,1}^{(1)}$, generate $I$. Thus it follows that $\gamma(h)\in Z(R)\sect (\partial I) K(R)$ for all $h\in \Sc$, and hence our assumption, $\partial(I)^2\subset I$, implies that $\gamma(h_1)\gamma(h_2)=0$ for any two elements $h_1,h_2\in \Sc$.
\end{proof}

\section{Classes of Golod ideals}
\label{classes}
  We keep the notation and the assumptions of Section~\ref{diff} and apply Theorem~\ref{muchbetter} to exhibit new classes of Golod rings.

It is customary to call a graded  ideal $I\subset S$ a {\em Golod ideal}, if $S/I$ is Golod. For convenience, we call  a graded ideal $I\subset S$ {\em strongly Golod}, if $\partial(I)^2\subset I$. As we have shown in Theorem~\ref{muchbetter}, any strongly Golod ideal is Golod.

As usual we denote by $I^{(k)}$ the symbolic powers  and by $\widetilde{I^k}$ the saturated powers of $I$. Recall that $I^{(k)}=\Union_{t\geq 1} I^k: L^t$,  where $L$ is the intersection of all associated, non-minimal  prime ideals of $I^k$, while   $\widetilde{I^k}=\Union_{t\geq 1} I^k: \mm^t$ where $\mm$ is the graded maximal ideal of  $S$.

Since $S$ is Noetherian, there exists an integer $t_0$ such that $I^{(k)}= I^k: L^t$ for all $t\geq t_0$. In particular, if we let  $J=L^t$ for some    $t\geq t_0$,  then $I^{(k)}=I:J=I:J^2$.

\medskip
The next result shows that strongly Golod ideals behave well with respect to several important ideal operations; in particular combining the various parts of
the theorem yields a quite large class of Golod ideals.

\begin{Theorem}
\label{twodaysbeforemy71thbirthday}
Let $I,J\subset S$ be graded ideals. Then the following hold:
\begin{enumerate}
\item[(a)] if $I$ and $J$ are strongly Golod, then $I\sect J$ and $IJ$ are strongly Golod;
\item[(b)] if $I$ and $J$ are strongly Golod and $\partial(I)\partial(J)\subset I+J$, then $I+J$ is strongly Golod;
\item[(c)] if  $I$ is strongly Golod,  $J$ is arbitrary,  and $I: J=I:J^2$, then $I:J$ is strongly Golod;
\item[(d)] $I^k$,  $I^{(k)}$ and  $\widetilde{I^k}$ are strongly Golod for all $k\geq 2$.
\end{enumerate}
\end{Theorem}

\begin{proof}
To simplify notation we write  $\partial f$ to mean anyone of the partials  $\partial f/\partial x_i$.

(a) Let $f,g\in I\sect J$. Since $I$ and $J$ are strongly Golod, it follows that $(\partial f)(\partial g)\in I$ and $(\partial f)(\partial g)\in J$, and hence   $(\partial f)(\partial g)\in I\sect J$. This shows that $I\sect J$ is strongly Golod.

Due to the product rule for partial derivatives it follows that $\partial(IJ)\subset \partial(I)J+I\partial(J)$. This implies that
\[
\partial(IJ)^2\subset \partial(I)^2J^2 + \partial(I)\partial(J)IJ+\partial(J)^2I^2.
\]
Obviously, the middle term is contained in $IJ$, while $\partial(I)^2J^2\subset IJ^2\subset IJ$, since $I$ is strongly Golod. Similarly, $\partial(J)^2I^2\subset IJ$. This shows that $\partial(IJ)^2\subset IJ$, and proves that $IJ$ is strongly Golod.

(b) is proved in the same manner as (a).

(c) Let $f,g\in I:J$. Then for all $h_1,h_2\in J$ one has $fh_1\in I$ and $gh_2\in I$. This implies that $(\partial f)h_1+f\partial h_1\in \partial(I)$ and
$(\partial g)h_2+g\partial h_2\in \partial(I)$. Thus
\begin{eqnarray*}
&&((\partial f)h_1+f\partial h_1)((\partial g)h_2+g\partial h_2)\\
&=&(\partial f)(\partial g)h_1h_2+(\partial f)(\partial h_2)gh_1+(\partial g)(\partial h_1)fh_2 +fg (\partial h_1)(\partial h_2)\in (\partial I)^2\subset I.
\end{eqnarray*}
Since $gh_1,fh_2\in I$, it follows that $(\partial f)(\partial h_2)gh_1+(\partial g)(\partial h_1)fh_2\in I$. Moreover,   $Jfg (\partial h_1)(\partial h_2)\subset I$.  Hence  $J(\partial f)(\partial g)h_1h_2\subset I$. Since $h_1,h_2$ were
arbitrary in $J$, it then follows that    $(\partial(I:J))^2\subset I:J^3=I:J$, as desired.

(d)  Let $k\geq 2$. Then $\partial(I^k)\subset I^{k-1}\partial(I)$. It follows that $\partial(I^k)^2\subset I^{2k-2}\partial(I)^2\subset I^k$. Thus $I^k$ is strongly Golod.

As explained above, for appropriately chosen $J$, $I^{(k)}=I^k:J=^kI:J^2$.  Now it follows from   (c) that $I^{(k)}$ is strongly Golod.

The same arguments show that $\widetilde{I^k}$ is strongly Golod.
\end{proof}

\begin{Corollary}
\label{mm}
Let $P$ be a homogeneous prime ideal of $S$ containing $I$, a  strongly Golod ideal. Then $I+P^k$ is strongly Golod for all $k\geq 2$.
\end{Corollary}

\begin{proof}
First note that  $\partial(I)\subset P$. If not, then  $\partial(I)^2$ will also not be in $P$, and therefore not in $I$, a contradiction. We also
observe that  $\partial(P^k)\subset P^{k-1}$. It follows that $(\partial I)(\partial P^k)\subset P^k\subset I+P^k$. Thus the assertion follows from Theorem~\ref{twodaysbeforemy71thbirthday}(b) and (d) (we apply (d) to ensure that $P^k$ is strongly Golod). \end{proof}

Let $R=S/I$  and  $\nn$ the graded maximal ideal of $R$, and suppose that $I$ is strongly Golod. Then Corollary~\ref{mm} implies that $R/\nn^k$ is Golod for all $k\geq 2$. Also note that by the theorem of Zariski-Nagata (see, e.g., \cite[p.143]{N}), the fact that $\partial(I)$ is contained in every homogeneous prime containing
$I$ implies that $I$ is in the second symbolic power of every such prime. In particular, self-radical ideals are never strongly Golod, though they may be Golod.

\begin{Corollary}
\label{components}
The (uniquely determined) primary components belonging to the minimal prime ideals  of a strongly Golod ideal are strongly Golod.
\end{Corollary}

\begin{proof}
Let $I$ be strongly Golod and  $P_1,\ldots,P_s$ its minimal prime ideals.  Let $Q_i$ be the primary component of $I$ with $\Ass(S/Q_i)=\{P_i\}$, and set $L_i=\Sect_{j\neq i}P_j$. Then there exists an integer $r>1$ such that $Q_i=I:L_i^r=I:L_i^{2r}$. It follows from  Theorem~\ref{twodaysbeforemy71thbirthday}(c)  that $Q_i$ is strongly Golod.
\end{proof}

\begin{Corollary}
\label{components1}
Every strongly Golod ideal has a primary decomposition with strongly Golod primary ideals.
\end{Corollary}

\begin{proof} Let $P$ be an associated prime of $I$. By Corollary \ref{mm},  $I+P^k$ is strongly Golod for all $k\geq 2$, and then by Corollary \ref{components}
the unique $P$-primary minimal component of $I+P^k$ is also strongly Golod. We denote this component by $P_k$. The Corollary now follows from the
general fact that $I = \cap P_k$ for large $k$, where the intersection is taken over all associated primes of $I$. To prove this, fix a primary decomposition
of $I$, say $I = \cap Q_P$, where $Q_P$ is $P$-primary, and the intersection runs over all associated primes of $I$. It suffices to prove that
$P_k\subset Q_P$ for large $k$, since then $I\subset \cap P_k\subset \cap Q_P = I$. To check that $P_k\subset Q_P$, we may localize at $P$, since
both of these ideals are $P$-primary. Then the claim is clear.
\end{proof}

Statement (d) of Theorem~\ref{twodaysbeforemy71thbirthday} can be substantially generalized follows.

\begin{Theorem}
\label{inbetween}
Let $I\subset S$  be a homogeneous ideal and suppose  suppose that $(I^{(k-1)})^2\subset I^k$ for some $k\geq 2$. Then all homogeneous ideals $J$ with $I^k\subset J\subset  I^{(k)}$ are strongly Golod. In particular, any homogeneous ideal $J$ with $I^2\subset J\subset I^{(2)}$ is strongly Golod.
\end{Theorem}

\begin{proof}
We use a theorem of Zariski-Nagata \cite[p.143]{N} according to which $f\in S$ belongs to $I^{(k)}$ if and only if all partials of $f$ of order $<k$ belong to $I$. It follows from this characterization of the $k$th symbolic power of $I$ that $\partial(I^{(k)})\subset I^{(k-1)}$.

Now assume that $I^k\subset J\subset I^{(k)}$. Then
\[
\partial(J)^2\subset \partial(I^{(k)})^2\subset (I^{(k-1)})^2\subset I^k\subset J.
\]
This proves that $J$ is strongly Golod.
\end{proof}

\begin{Example} {\rm Let $X$ be a generic set of points in $\mathbb P^2$, and let $I$ be the ideal of polynomials vanishing at $X$. Bocci and Harbourne \cite{BH} proved  that  $I^{(3)}\subset I^2$. In this case, $(I^{(3)})^2\subset I^4$, so Theorem \ref{inbetween} applies to conclude that every homogeneous ideal
$J$ between $I^3$ and $I^{(3)}$ is strongly Golod.} \end{Example}

\begin{Example} {\rm The hypothesis of Theorem \ref{inbetween} is certainly not always satisfied. For example, let $X$ be a generic $4$ by $4$ matrix,
and let $I$ be the ideal generated by the  $3$ by $3$ minors of $X$. It is well-known that $I$ is a height 4 prime ideal. Moreover, if $\Delta$ is the
determinant of $X$, then $\Delta\in I^{(2)}$. However,  $\Delta^2$ cannot be in $I^3$ since this element has degree $8$, and the generators of
$I^3$ have degree $9$. Thus $(I^{(2)})^2$ is not contained in $I^3$.} \end{Example}

Another case to consider when   the hypotheses of Theorem~\ref{inbetween} are satisfied is the following: Let $G$ e a finite simple graph on the vertex set $[n]$. A vertex cover of $G$ is a subset $C\subset [n]$ such that  $C\sect \{i,j\}\neq \emptyset$ for all edges $\{i,j\}$ of $G$.  The vertex cover ideal $I$ of $G$ is the ideal generated by all monomials $\prod_{i\in C}x_i\subset S=K[x_1,\ldots,x_n]$ where $C$ is a vertex cover of $G$. It is obvious that
\[
I= \Sect_{\{i,j\}\in E(G)}(x_i,x_j),
\]
where $E(G)$ denotes the set of edges of $G$.
It follows that $I^{(k)}=\Sect_{\{i,j\}\in E(G)}(x_i,x_j)^k$ for all $k$.

\begin{Proposition}
\label{onecase}
Let $I$ be the vertex cover ideal a finite simple graph $G$, and suppose that $(I^{(2)})^2\subset I^3$. Then  $(I^{(k-1)})^2\subset I^k$ for all $k$. In particular,
every homogeneous ideal $J$ between $I^k$ and $I^{(k)}$ is strongly Golod.
\end{Proposition}

In \cite[Theorem~5.1]{HHT} it is shown that the graded $S$-algebra $\Dirsum_{k\geq 0}I^{(k)}t^k\subset S[t]$ is  generated in degree $1$ and $2$.  Thus
the proposition follows from the more general proposition:

\begin{Proposition}
\label{onecase}
Let $I$ be an unmixed ideal of height two,  and suppose that  $\Dirsum_{k\geq 0}I^{(k)}t^k$ is generated in degrees $1$ and $2$, and  $(I^{(2)})^2\subset I^3$. Then  $(I^{(k-1)})^2\subset I^k$ for all $k$.
\end{Proposition}

\begin{proof} The assumption that $\Dirsum_{k\geq 0}I^{(k)}t^k\subset S[t]$ is  generated in degree $1$ and $2$ is equivalent to
the statement that for all $p\geq 0$ one has
\[
I^{(2p+1)}=I(I^{(2)})^{p}\quad \text{and}\quad I^{(2p)} = (I^{(2)})^{p}.
\]
Since $I$ is an ideal of height $2$ it follows by a theorem of Ein,  Lazarsfeld, and Smith \cite{ELS} that $I^{(2p)}\subset I^p$ for all $p\geq 0$.
This implies that $(I^{(2)})^p\subset I^p$ for all $p$, since $(I^{(2)})^p\subset I^{(2p)}$.

We want to prove that $(I^{(k-1)})^2\subset I^k$. We consider the two cases where $k$ is even or odd.

We have that
\[
(I^{(2p+1)})^2=I^{2}((I^{(2)})^p)^2\subset I^{2}(I^{p})^2=I^{2p+2}.
\]

Next assume that $k$ is even. Then $(I^{(2p)})^2=(I^{(2)})^{2p}$. Thus
\[
(I^{(2p)})^2=((I^{(2)})^{2})^p\subset (I^3)^p\subset I^{2p+1}.
\]
 Here we used that $(I^{(2)})^{2}\subset I^3$.
\end{proof}

Proposition~\ref{onecase} is trivial if $G$ does not contain an odd cycle, in other words, if $G$ is bipartite, because in this case it is known
\cite[Theorem~5.1(b)]{HHT} that the symbolic and ordinary powers of $I$ coincide. The first non-trivial case is that when $G$ is an odd cycle, say on the vertex set $[n]$. In that case it is known, and easy to see,  that $I$ is generated by the monomials $u_i=\prod_{j=i}^{(n+1)/2}x_{i+2j}$, where for simplicity of notation $x_i=x_{i-n}$ if $i>n$. It is also known that $I^{(2)}=I^2+(u)$, where $u=\prod_{i=1}^nx_i$, see for example \cite[Proposition~5.3]{HHT}. It follows that $(I^{(2)})^2=I^4+I^2(u)+(u^2)$. Since $u\in I$ we see that $I^2(u)\subset I^3$, and since $u_1u_2u_3$ divides $u^2$ we also have $(u^2)\subset I^3$. So that altogether, $(I^{(2)})^2\subset I^3$ for the vertex cover ideal of any odd cycle. We do not whether the same inclusion holds for any graph containing and odd cycle.

\begin{Example}
{\em There  exist squarefree monomial ideals $I$ for which  $(I^{(2)})^2$ is not contained in  $I^3$. Indeed, let $I_{n,d}$ be  the ideal generated by all squarefree monomials of degree $d$ in $n$ variables, and let $2<d<n$. Then $u=\prod_{i=1}^{d+1}x_i\in I_{n,d}^{(2)}$, but $u^2\not \in I_{n,d}^3$, because $I_{n,d}^3$ is generated in degree $3d$, while $\deg u^2=2(d+1)<3d$.
}
\end{Example}

\medskip
It is immediate that a monomial ideal $I$ is strongly Golod, if for all minimal monomial generators $u,v \in I$ and all integers $i$ and $j$ such that  $x_i|u$ and $x_j|v$ it follows that $uv/x_ix_j\in I$.  Farkhari and Welker showed \cite{FW} that $IJ$ is Golod for any two proper monomial ideals. However a product of proper monomial ideals need not to be strongly Golod as the following example shows: Let $I=(x,y)$ and $J=(z)$. Then $IJ=(xz,yz)$ and $(xz)(yz)/xy=z^2$ does not belong to $IJ$.

We denote by $\bar{I}$ the integral closure of an ideal, and use the above characterization of strongly Golod monomial ideals to show:

\begin{Proposition}
\label{integral}
Let $I$ be a strongly Golod monomial ideal. Then $\bar{I}$ is strongly Golod. In particular, for any monomial ideal $I$, the ideals $\overline{I^k}$ are strongly Golod for all $k\geq 2$.
\end{Proposition}

\begin{proof}

The integral closure of the monomial ideal $I$ is again a monomial ideal, and a monomial $u\in S$ belongs to $\bar{I}$ if and only if there exists an integer $r>0$ such that $u^r\in I^r$, see \cite[Proposition 1.4.2]{SwaHu} and the remarks preceding this proposition.

Let $u,v\in \bar{I}$ monomials and suppose that $x_i|u$ and $x_j|v$. There exist integers $s,t>0$ such that $u^s\in I^s$ and $v^t\in I^t$. We may assume that both, $s$ and $t$, are even.

We observe that for a  monomial $w$ with $x_k|w$ and  $w^r\in I^r$ it follows that $(w/x_k)^r\in I^{r/2}$ if $r$ is even.  Indeed, $w^r=m_1\cdots m_r$ with $m_i\in I$. Let $d_i$ be the highest exponent that divides $m_i$. We may assume that $d_i=0$ for $i=1,\ldots, a$, $d_i=1$ for $i=a+1,\ldots,b$ and $d_i>1$ for $i>b$. Then
\[
(w/x_k)^r= (m_1\cdots m_a)((m_{a+1}/x_k)\cdots (m_b/x_k))((m_{b+1}\cdots m_r)/x_k^d),
\]
where $d=\sum_{i=b+1}^rd_i$.

Since $I$ is strongly Golod, it follows  from this presentation of $(w/x_k)^r$ that
\[
(w/x_k)^r\in I^aI^{\lfloor b/2\rfloor}=I^{\lfloor(2a+b)/2\rfloor},
\]
where $\lfloor c\rfloor$ denotes the lower integer part of the real number $c$.

Note that $r=d+b$ and that $2(r-a-b)\leq d$. This implies that  $2a\geq d$, and hence $2a+b\geq d+b=r$,  which implies that $(w/x_k)^r\in I^{r/2}$, as desired.

Now it follows that
\[
(uv/x_ix_j)^{s+t}=(u/x_i)^{t+s}(v/x_j)^{t+s}\in I^{s+t}
\]
This shows that $uv/x_ix_j\in\bar{I}$, and proves that  $\bar{I}$ is strongly Golod.
\end{proof}

\end{document}